\documentclass[11pt,reqno,final]{amsart} 

\usepackage{amssymb,latexsym}
\usepackage[draft=true]{hyperref}
\usepackage{cite} 

\usepackage[height=190mm,width=130mm]{geometry} 

\allowdisplaybreaks

\theoremstyle{plain} 
\newtheorem{theorem}{Theorem} 
\newtheorem{lemma}{Lemma}

\theoremstyle{definition} 
\newtheorem{definition}{Definition} 
\theoremstyle{remark} 
\newtheorem{remark}{Remark}

\AtBeginDocument{{\noindent\small
Preprint whose final and definite form is with 
\href{http://mat76.mat.uni-miskolc.hu/~mnotes}{\emph{Miskolc Math. Notes}}, 
ISSN 1787-2405.}
\vspace{2mm}}

\begin{document}

\title[Approximate Controllability and Optimal Control]{Approximate 
Controllability of Impulsive\\ 
Non-local Non-linear Fractional\\ 
Dynamical Systems and Optimal Control}


\author[S. Guechi]{Sarra Guechi} 
\address{Guelma University\\ Department of Mathematics\\ 24000 Guelma\\ Algeria} 
\email{guechi.sara@yahoo.fr} 

\author[A. Debbouche]{Amar Debbouche} 
\address{Guelma University\\ Department of Mathematics\\ 24000 Guelma\\ Algeria} 
\email{amar$_{-}$debbouche@yahoo.fr} 

\author[D. F. M. Torres]{Delfim F. M. Torres} 
\address{University of Aveiro\\ DMat\\ 
CIDMA\\ 3810-193 Aveiro\\ Portugal}
\email{delfim@ua.pt} 


\begin{abstract}
We establish existence, approximate controllability and optimal control 
of a class of impulsive non-local non-linear fractional dynamical systems 
in Banach spaces. We use fractional calculus, sectorial operators and 
Krasnoselskii fixed point theorems for the main results. Approximate 
controllability results are discussed with respect to the 
inhomogeneous non-linear part. Moreover, we prove existence results 
of optimal pairs of corresponding fractional control systems with a Bolza cost
functional. 
\end{abstract}



\maketitle


\section{Introduction}
We are concerned with an impulsive non-local non-linear 
fractional control dynamical system of form
\begin{equation} 
\label{e1}
\left\lbrace 
\begin{array}{lll} 
^{C}D_{t}^{q}x(t)=Ax(t)+f(t, x(t), (Hx)(t))+Bu(t), 
\ t\in (0, b]\setminus\lbrace t_1, t_2,\ldots, t_m\rbrace,\\
x(0)+g(x)=x_{0}\in X,\quad
\triangle x(t_{i})=I_{i}(x(t_{i}^{-}))+Dv(t_{i}^{-}), \quad i=1,2,\ldots,m,
\end{array}
\right.
\end{equation}
where $^{C}D_{t}^{q}$ is the Caputo fractional derivative of order $0<q<1$, 
the state $x(\cdot)$ takes its values in a Banach space $X$ with norm 
$\Vert\cdot\Vert$, and $x_0\in X$. Let $A: D (A)\subset X\to X$ be a sectorial 
operator of type $(M, \theta, q, \mu)$ on $X$, $H: I\times I\times X \to X$ 
represents a Volterra-type operator such that $(Hx)(t)=\int_{0}^{t}h(t, s, x(s))ds$, 
the control functions $u(\cdot)$ and $v(\cdot)$ are given in $L^{2}(I, U)$, 
$U$ is a Banach space, $B$ and $D$ are bounded linear operators from $U$ into $X$. 
Here, one has $I=[0, b]$, $0=t_0<t_1<\cdots<t_m<t_{m+1}=b$, $I_{i}: X\to X$ are impulsive 
functions that characterize the jump of the solutions at impulse points $t_i$, 
the non-linear term $f: I \times X\times X\to X$, the non-local function 
$g: PC(I, X)\to X$, with $PC$ defined later, $\triangle x(t_{i})=x(t_i^+)-x(t_i^-)$, 
where $x(t_i^+)$ and $x(t_i^-)$ are the right and left limits of $x$ at the point $t_i$, 
respectively. 

Derivatives and integrals of arbitrary order, the main objects of Fractional Calculus (FC), 
have kept the interest of many scientists in recent years, since they 
provide an excellent tool to describe hereditary properties of various materials 
and processes. During the past decades, FC and its applications have gained 
a lot of importance, due to successful results in modelling several complex phenomena 
in numerous seemingly diverse and widespread 
fields of science and engineering, such as heat conduction, diffusion, 
propagation of waves, radiative transfer, kinetic theory of gases, diffraction 
problems and water waves, radiation, continuum mechanics, geophysics, electricity 
and magnetism, as well as in mathematical economics, communication theory, 
population genetics, queuing theory and medicine. For details on the theory and 
applications of FC see \cite{r3}. For recent developments 
in non-local and impulsive fractional differential problems see 
\cite{r6,r7,r8,r9} and references therein.

The problem of controllability is one of the most important qualitative 
aspects of dynamical systems in control theory. It consists to show the 
existence of a control function that steers the solution of the system from 
its initial state to a final state, where the initial and final states may 
vary over the entire space. This concept plays a major role in finite-dimensional 
control theory, so that it is natural to try to generalize it to infinite 
dimensions \cite{r12}. Moreover, exact controllability for semi-linear fractional 
order systems, when the non-linear term is independent of the control function, 
is proved by assuming that the controllability operator has an induced inverse 
on a quotient space. However, if the semi-group associated with the system is compact, 
then the controllability operator is also compact and hence the induced inverse 
does not exist because the state space is infinite dimensional \cite{r13}. 
Thus, the concept of exact controllability is too strong and
has limited applicability, while approximate controllability 
is a weaker concept completely adequate in applications.

On the other hand, control systems are often based on the principle 
of feedback, where the signal to be controlled is compared to a desired 
reference, and the discrepancy is used to compute a corrective control action. 
Fractional optimal control of a distributed system is an optimal control problem
for which the system dynamics is defined with fractional differential equations. 
Recently, attention has been paid to prove existence, approximate controllability 
and/or optimal control for different classes of fractional differential equations
\cite{r14,r15,r16,r17}. 

In \cite{r18}, optimal control of non-instantaneous impulsive differential equations is studied.
Qin et al. investigate approximate controllability and optimal control of
fractional dynamical systems of order $1<q<2$ in Banach spaces \cite{r19}. 
Debbouche and Antonov  established approximate controllability of semi-linear Hilfer 
fractional differential inclusions with impulsive control inclusion 
conditions in Banach spaces \cite{r20}. Motivated by the above works, 
here we construct an impulsive non-local non-linear 
fractional control dynamical system and prove new sufficient conditions 
to treat the questions of approximate controllability and optimal control.

The paper is organized as follows. In Section~\ref{sec:2}, we recall some 
facts from fractional calculus, $q$-resolvent families, and useful versions 
of fixed point techniques that are used for obtaining our main results. 
In Section~\ref{sec:3}, we form appropriate sufficient conditions and prove 
existence results for the fractional control system \eqref{e1}. 
In Section~\ref{sec:4}, we investigate the question of
approximate controllability. We end with Section~\ref{sec:5}, 
where we obtain optimal controls corresponding to fractional 
control systems with a Bolza cost functional. 


\section{Preliminaries}
\label{sec:2}

Here we present some preliminaries from fractional calculus \cite{r3}, 
operator theory \cite{r22} and fixed point techniques \cite{r6}, 
which are used throughout the work to obtain the desired results. 

\begin{definition}
\label{Definition:2.1}
The left-sided Riemann--Liouville fractional integral of order $\alpha>0$, 
with lower limit $a$, for a function $f: [a, +\infty)\to \mathbb{R}$, is defined as
$$
I^{\alpha}_{a^{+}}f(t)=\frac{1}{\Gamma(\alpha)}\int_{a}^{t}(t-s)^{\alpha-1}f(s)ds,
$$
provided the right side is point-wise defined on $[a, +\infty)$, 
where $\Gamma(\cdot)$ is the Euler gamma function. If $a=0$, 
then we can write $I^{\alpha}_{0^{+}}f(t)=(g_{\alpha}*f)(t)$, where
$$
g_{\alpha}(t):=\left\{
\begin{array}{ll}
\frac{1}{\Gamma(\alpha)}t^{\alpha-1},
& \mbox{$t>0$},\\0, & \mbox{$t\leq 0$},
\end{array}\right.
$$
and $*$ denotes convolution of functions. Moreover, 
$\lim\limits_{\alpha\rightarrow 0}g_{\alpha}(t)=\delta(t)$, 
with $\delta$ the delta Dirac function.
\end{definition}

\begin{definition}
\label{Definition:2.2}
The left-sided Riemann--Liouville fractional derivative of order $\alpha>0$, 
$n-1\leq\alpha<n$, $n\in \mathbb{N}$, for a function $f: [a, +\infty)\to \mathbb{R}$, 
is defined by
$$
^{L}D^{\alpha}_{a^{+}}f(t)=\frac{1}{\Gamma(n-\alpha)}
\frac{d^{n}}{dt^{n}}\int_{a}^{t}\frac{f(s)}{(t-s)^{\alpha+1-n}}ds,
\quad t>a,
$$
where function $f$ has absolutely continuous derivatives up to order $n-1$.
\end{definition}

\begin{definition}
\label{Definition:2.3}
The left-sided Caputo fractional derivative of order $\alpha>0$, 
$n-1<\alpha<n$, $n\in \mathbb{N}$, for a function 
$f: [a, +\infty)\to \mathbb{R}$, is defined by
$$
^{C}D^{\alpha}_{a^{+}}f(t)=\frac{1}{\Gamma(n-\alpha)}
\int_{a}^{t}\frac{f^{(n)}(s)}{(t-s)^{\alpha+1-n}}ds
=I^{n-\alpha}_{a^{+}}f^{(n)}(t),\quad  t>a,
$$
where function $f$ has absolutely continuous derivatives up to order $n-1$.
\end{definition}

Throughout the paper, by $PC(I, X)$ we denote the space of $X$-valued bounded 
functions on $I$ with the uniform norm $\Vert x\Vert_{PC}=\sup\lbrace 
\Vert x(t)\Vert, t\in I\rbrace$ such that $x(t^{+}_{i})$ exists for 
any $i=0,\ldots,m$ and $x(t)$ is continuous on $(t_i, t_{i+1}]$, 
$i=0,\ldots,m$, $t_0=0$ and $t_{m+1}=b$.

\begin{definition}[See \cite{r21}]
\label{Definition:2.4}
Let $A: D\subseteq X \to X$ be a closed and linear operator. 
We say that $A$ is \emph{sectorial} of type $(M,\theta,q ,\mu)$, 
if there exists $\mu\in \mathbb{R}$, $0<\theta<\frac{\pi}{2}$
and $M>0$ such that the $q$-resolvent 
of $A$ exists outside the sector 
$$
\mu+S_{\theta}=\lbrace\mu+\lambda^{q}: 
\lambda\in \mathbb{C},\vert\text{Arg}(-\lambda^{q})\vert <\theta\rbrace
$$
and
$$
\Vert(\lambda^{q} I-A)^{-1}\Vert 
\leq \frac{M}{\vert \lambda^{q}-\mu\vert},\lambda^{q}\notin\mu+S_{\theta}.
$$
\end{definition}

\begin{remark}
\label{Remark:2.5}
If $A$ is a sectorial operator of type $(M, \theta, q, \mu)$, 
then it is not difficult to see that $A$ is the infinitesimal 
generator of a $q$-resolvent family ${T_{q}(t)}_{t\geq 0}$ in a Banach space, 
where $T_{q}(t)=\frac{1}{2\pi i}\int_{c}e^{\lambda t}R(\lambda^{q},A)d\lambda$.
\end{remark}

\begin{definition} [Motivated by \cite{r20,r21}]
\label{Definition:2.6}
A state function $x\in PC(I, X)$ is called a mild solution of \eqref{e1} 
if it satisfies the following integral equations:
\begin{equation*}
x(t)= S_{q}(t)(x_{0}-g(x))+\int_{0}^{t}T_{q}(t-s)(f(s,x(s),(Hx)(s))
+Bu(s))ds
\end{equation*}
if $t\in[0, t_{1}]$, and
\begin{multline*}
x(t)= S_{q}(t-t_{i})[x(t_i^-)+I_{i}(x(t_{i}^{-}))+Dv(t_{i}^{-})]\\
+\int_{t_{i}}^{t}T_{q}(t-s)[f(s,x(s),(Hx)(s))+Bu(s)]ds
\end{multline*}
if $t\in(t_{i}, t_{i+1}]$, $i=1,\ldots,m$, where
\begin{equation*}
S_{q}(t)=\frac{1}{2\pi i}\int_{c}e^{\lambda t}\lambda^{q-1}R(\lambda^{q},A)d\lambda
\quad \text{and} \quad
T_{q}(t)=\frac{1}{2\pi i}\int_{c}e^{\lambda t}R(\lambda^{q},A)d\lambda
\end{equation*}
with $c$ being a suitable path such that 
$\lambda^{q}\notin \mu+S_{\theta}$ for $\lambda\in c$.
\end{definition}

Let $x_{t_{k}}(x(0), \triangle x(t_{k-1}); u, v)$, 
$k=1,\ldots,m+1$, be the state value of \eqref{e1}
at time $t_{k}$, corresponding to the non-local 
initial value $x(0)$, the impulsive values 
$\triangle x(t_{k-1})=x(t_{k-1}^+)-x(t_{k-1}^-)$ and 
the controls $u$ and $v$. For every $x(0)$
and $\triangle x(t_{k-1})\in X$, we introduce the set
$$
\mathfrak{R}(t_{k}, x(0), \triangle x(t_{k-1}))=\left\lbrace
x_{t_{k}}\left(x(0), \triangle x(t_{k-1}); u, v\right): 
u(\cdot),v(\cdot)\in L^{2}(I, U)\right\rbrace,
$$
which is called the \emph{reachable set} of system \eqref{e1}
at time $t_{k}$ (if $k=m+1$, then $t_{k}$ is the terminal time). 
Its closure in $X$ is denoted by 
$\overline{\mathfrak{R}(t_{k}, x(0), \triangle x(t_{k-1}))}$.

\begin{definition}
\label{Definition:2.7}
The impulsive control system \eqref{e1} is said 
to be approximately controllable on $I$
if $\overline{\mathfrak{R}(t_{k}, x(0), \triangle x(t_{k-1}))}=X$, 
that is, given an arbitrary $\epsilon>0$,
it is possible to steer from the points $x(0)$ and $\triangle x(t_{k-1})$ 
at time $t_{k}$ all points in the state space
$X$ within a distance $\epsilon$.
\end{definition}

Consider the linear impulsive fractional control system
\begin{equation*}
\label{e3}
\left\lbrace 
\begin{array}{lll} 
^{C}D_{t}^{q}x(t)=Ax(t)+Bu(t),\\
x(0)=x_{0}\in X,\\
\triangle x(t_{i})=Dv(t_{i}^{-}), \quad i=1,\ldots,m.
\end{array}
\right. 
\end{equation*}
Approximate controllability for the linear impulsive 
control fractional system \eqref{e3} is a natural
generalization of the notion of approximate controllability
of a linear first-order control system ($q=1$ and $t_i=D=0$, 
$i=1,2,\ldots,m$, i.e., $t\in [t_{m}, t_{m+1}]=[0, b]$).
The controllability operators
associated with \eqref{e3} are
\begin{equation}
\label{e4}
\begin{aligned}
\Psi^{t_k}_{t_{k-1}, 1}&=\int_{t_{k-1}}^{t_k}
T_{q}(t_k-s)BB^{\ast}T_{q}^{\ast}(t_k-s)ds,\quad k=1,\ldots,m+1,\\
\Psi^{t_k}_{t_{k-1}, 2}&=S_{q}(t_k-t_{k-1})DD^{\ast}S_{q}^{\ast}(t_k-t_{k-1}),
\quad k=2,\ldots,m+1,
\end{aligned}
\end{equation}
where $T_{q}^{\ast}(\cdot)$, $S_{q}^{\ast}(\cdot)$, $B^{\ast}$ and $D^{\ast}$
denote the adjoints of $T_{q}(\cdot)$, $S_{q}(\cdot)$, $B$ and $D$, respectively.
Moreover, for $\lambda>0$, we consider the relevant operator
\begin{equation}
\label{e5}
\mathcal{R}(\lambda, \Psi^{t_k}_{t_{k-1}, i})
=\left(\lambda I+\Psi^{t_k}_{t_{k-1}, i}\right)^{-1},
\quad i=1,2.
\end{equation}
It is easy to verify that
$\Psi^{t_k}_{t_{k-1}, 1}$ and 
$\Psi^{t_k}_{t_{k-1}, 2}$ 
are linear bounded operators.

\begin{lemma}[See \cite{r20}]
\label{Lemma:2.8}
The linear impulsive control fractional system \eqref{e3}
is approximately controllable on $I$ if and only if
$\lambda\mathcal{R}(\lambda, \Psi^{t_k}_{t_{k-1}, i})\rightarrow 0$ 
as $\lambda\rightarrow 0^{+}$, $i=1,2$,
in the strong operator topology.
\end{lemma}

\begin{lemma}[Krasnoselskii theorem \cite{r23}]
\label{Lemma:2.9}
Let $X$ be a Banach space and $E$ be a
bounded, closed, and convex subset of $X$. 
Let $Q_{1}, Q_{2}$ be maps of $E$ into $X$ such that
$Q_{1}x + Q_{2}y \in E$ for every $x, y \in E$. 
If $Q_{1}$ is a contraction and $Q_{2}$ is compact and continuous,
then equation $Q_{1}x + Q_{2}x = x$ has a solution on $E$.
\end{lemma}


\section{Existence of a mild solution}
\label{sec:3}

We prove existence for system \eqref{e1}. 
Define $K_i^*= \sup_{t\in I} \int_{t_{i-1}}^{t_{i}}m(t, s) ds<\infty$, 
$i=1,\dots,m+1$. For any $r > 0$, let 
$\Omega_{r} := \lbrace x \in PC(I, X)\vert \|x\|\leq r\rbrace$. 
We make the following assumptions:
\begin{itemize}
\item[(H$_1$)] The operators $S_{q}(t)_{t\geq0}$ and $T_{q}(t)_{t\geq0}$, 
generated by $A$, are bounded and compact, such that
$\sup_{t\in I}\|S_{q}(t)\|\leq M$ and $\sup_{t\in I}\|T_{q}(t)\|\leq M$.

\item[(H$_2$)] The non-linearity $f: I\times X\times X\to X$ 
is continuous and compact; there exist functions $\mu_{i}\in
L^{\infty}(I, \mathbb{R}^{+})$, $i=1,2,3$, and positive constants 
$\alpha_{1}$ and $\alpha_{2}$ such that
$\|f (t, x, y)\| \leq \mu_{1}(t) + \mu_{2}(t)\|x\| + \mu_{3}(t)\|y\|$
and
$\|f (t, x, Hx) - f (t, y, Hy)\| = \alpha_{1}\|x - y\| + \alpha_{2}\|Hx - Hy\|$.

\item[(H$_3$)] Function $g: PC(I, X)\to X$ is completely 
continuous and there exists a positive
constant $\beta$ such that
$\|g(x)-g(y)\|\leq\beta\|x-y\|$, $x,y\in X$.

\item[(H$_4$)] Associated with $h: \Delta\times X\to X$, there exists 
$m(t, s)\in PC(\Delta, \mathbb{R}^{+})$ such that
$\|h(t, s, x(s))\| \leq m(t, s)\|x\|$
for each $(t, s)\in\Delta$ and $x,y\in X$, where 
$\Delta=\lbrace (t, s)\in \mathbb{R}^{2}\vert t_i
\leq s$, $t\leq t_{i+1}$, $i=0,\ldots,m\rbrace$.

\item[(H$_5$)] For every $x_1, x_2, x\in X$ and $t\in (t_i, t_{i+1}]$, 
$i=1,\ldots,m$, $I_i$ are continuous and compact and there exist 
positive constants $d_i$, $e_i$ such that
$$
\Vert I_i(x_1(t_i^-))-I_i(x_2(t_i^-))\Vert
\leq d_i\sup\limits_{t\in (t_i, t_{i+1}]}\Vert x_1(t)-x_2(t)\Vert 
$$
and
$\Vert I_i(x(t_i^-))\Vert\leq e_i\sup\limits_{t\in (t_i, t_{i+1}]}\Vert x(t)\Vert$.
\end{itemize}

\begin{theorem}
\label{Theorem:2.10}
Let $x_{0}\in X$. If conditions (H$_1$)--(H$_5$) hold, 
then the impulsive non-local fractional control system \eqref{e1} 
has a fixed point on $I$ provided $M\beta<1$ 
and $M(1+d_{i})<1$, $i=1,\ldots,m$, that is,
\eqref{e1} has at least one mild solution
on $t\in [0, b]\setminus\lbrace t_1,\ldots,t_m\rbrace$.
\end{theorem}

\begin{proof}
Define the operators $Q_{1}$ and $Q_{2}$ on $\Omega_{r}$ as follows:
\begin{equation*}
(Q_{1}x)(t)=\left\lbrace 
\begin{array}{ll}
S_{q}(t)(x_{0}-g(x)), \quad t\in[0,t_{1}]\\
S_{q}(t-t_{i})[x(t_{i}^{-})+I_{i}(x(t_{i}^{-}))+Dv(t_{i}^{-})], 
\quad  t\in(t_{i},t_{i+1}], 
\end{array}\right.
\end{equation*}
\begin{equation*}
(Q_{2}x)(t)=\left\lbrace 
\begin{array}{ll}
\int_{0}^{t} T_{q}(t-s))(f(s,x(s),(Hx)(s))+Bu(s))ds,  
\quad t\in[0,t_{1}],\\
\int_{t_{i}}^{t} T_{q}(t-s))(f(s,x(s),(Hx)(s))+Bu(s))ds, 
\quad t\in(t_{i},t_{i+1}],
\end{array}
\right. 
\end{equation*}
$i=1,\ldots,m$. We take the controls
\begin{equation}
\label{e6}
\begin{split}
u&=B^{\ast}T_{q}^{\ast}(t_k-t)
\mathcal{R}(\lambda, \Psi^{t_k}_{t_{k-1},1})P_{1}^{k}(x(\cdot)),\\
v&=D^{\ast}S_{q}^{\ast}(t_k-t_{k-1})
\mathcal{R}(\lambda, \Psi^{t_k}_{t_{k-1},2})P_{2}^{k}(x(\cdot)),
\end{split}
\end{equation}
where
\begin{equation*}
P_{1}^{k}\left(x(\cdot)\right)
=
\begin{cases}
x_1-S_{q}(t_1)(x_0-g(x))\\
~~~~~-\int_{0}^{t_1}T_{q}(t_1-s)f(s, x(s), (Hx)(s))ds,\quad  k=1,\\[0.3cm]
x_k-S_{q}(t_k-t_{k-1})[x(t_{k-1}^-)+I_{k-1}(x(t_{k-1}^-))]\\
~~~~~-\int_{t_{k-1}}^{t_k}T_{q}(t_k-s)f(s, x(s), (Hx)(s))ds,
\quad  k=2,\ldots,m+1,
\end{cases}
\end{equation*}
\begin{equation*}
P_{2}^{k}\left(x(\cdot)\right)
=
\begin{cases}
x_k-S_{q}(t_k-t_{k-1})[x(t_{k-1}^-)+I_{k-1}(x(t_{k-1}^-))]\\
~~~~~-\int_{t_{k-1}}^{t_k}T_{q}(t_k-s)f(s, x(s), (Hx)(s))ds,
\quad k=2,\ldots,m+1.
\end{cases}
\end{equation*}
For any $\lambda>0$, we shall show that $Q_{1}+Q_{2}$ has a fixed point 
on $\Omega_{r}$, which is a solution of system \eqref{e1}. According 
to \eqref{e6}, together with \eqref{e4} and \eqref{e5}, we have
\begin{equation}
\label{e7}
\|u(t)\|\leq \dfrac{1}{\lambda}M\Vert B\Vert\|P_{1}(x(\cdot))\| 
~\text{and} ~ \|v(t)\|\leq \dfrac{1}{\lambda}M\Vert D\Vert\|P_{2}(x(\cdot))\|.
\end{equation}
Using assumptions $(H_{1})$--$(H_{5})$, we get
\begin{equation*}
\begin{split}
\|P^1_{1}(x(\cdot))\|
&\leq \|x_{1}\| +\| S_{q}(t_{1})\|\|(x_{0} - g(x))\|\\
&\quad +\int_{0}^{t_{1}}\|T_{q}(t_{1}-s)\|\|f(s,x(s),(Hx)(s))\|ds\\
&\leq \|x_{1}\| +M(\|x_{0}\| +\| g(x)\|)\\
&\quad +Mt_{1}\left(\|\mu_{1}\|_{L^{\infty}(I, \mathbb{R}^{+})} 
+ r\|\mu_{2}\|_{L^{\infty}(I, \mathbb{R}^{+})} 
+ K_{1}^{*}r\|\mu_{3}\|_{L^{\infty}(I, \mathbb{R}^{+})}\right)\\
&\leq \|x_{1}\| +M\|x_{0}\| +M\beta \|x\| +M\Vert g(0)\Vert\\
&\quad +Mt_{1}\left(\|\mu_{1}\|_{L^{\infty}(I, \mathbb{R}^{+})} 
+ r\|\mu_{2}\|_{L^{\infty}(I, \mathbb{R}^{+})} 
+ K_{1}^{*}r\|\mu_{3}\|_{L^{\infty}(I, \mathbb{R}^{+})}\right)
\end{split}
\end{equation*}
and, for $k=2,\ldots,m+1$,
\begin{equation*}
\begin{split}
\|P^k_{1}(x(\cdot))\|
&\leq 
\|x_{k}\| +\| S_{q}(t_{k}-t_{k-1})\|[\Vert x(t_{k-1}^{-})\Vert + \|I_{k-1}(x(t_{k-1}^{-}))\|]\\
&\quad +\int_{t_{k-1}}^{t_{k}}\|T_{q}(t_{k}-s)\|\|f(s,x(s),(Hx)(s))\|ds\\
&\leq \|x_{k}\| +M(\Vert x(t_{k-1}^{-})\Vert+e_{i}\|x\|) + M(t_{k}-t_{k-1})\\
&\quad \times \left(\|\mu_{1}\|_{L^{\infty}(I, \mathbb{R}^{+})} 
+ r\|\mu_{2}\|_{L^{\infty}(I, \mathbb{R}^{+})} + K_{k}^{*}r\|
\mu_{3}\|_{L^{\infty}(I, \mathbb{R}^{+})}\right)\\
&\leq \|x_{k}\| +M(\Vert x(t_{k-1}^{-})\Vert+re_{i}) +M(t_{k}-t_{k-1})\\
&\quad \times \left(\|\mu_{1}\|_{L^{\infty}(I, \mathbb{R}^{+})} 
+ r\|\mu_{2}\|_{L^{\infty}(I, \mathbb{R}^{+})} 
+ K_{k}^{*}r\|\mu_{3}\|_{L^{\infty}(I, \mathbb{R}^{+})}\right).
\end{split}
\end{equation*}
Similarly, we get
\begin{multline*}
\|P^k_{2}(x(\cdot))\|\leq 
\|x_{k}\| +M(\Vert x(t_{k-1}^{-})\Vert+re_{k-1})\\
+M(t_{k}-t_{k-1})\left(\|\mu_{1}\|_{L^{\infty}(I, \mathbb{R}^{+})} 
+ r\|\mu_{2}\|_{L^{\infty}(I, \mathbb{R}^{+})} 
+ K_{k}^{*}r\|\mu_{3}\|_{L^{\infty}(I, \mathbb{R}^{+})}\right),
\end{multline*}
$k=2,\ldots,m+1$. For any $x \in \Omega_{r}$, we obtain 
\begin{equation*}
\begin{split}
\|(Q_{1}&x)(t) + (Q_{2}x)(t)\|\\ 
&\leq M\left(\|x_{0}\|+\|g(x)\|\right)\\
&\quad + Mt_{1}\left(\|\mu_{1}\|_{L^{\infty}(I, \mathbb{R}^{+})} 
+ r\|\mu_{2}\|_{L^{\infty}(I, \mathbb{R}^{+})} 
+ K_{1}^{*}r\|\mu_{3}\|_{L^{\infty}(I, \mathbb{R}^{+})}
+\Vert B\Vert \Vert u\Vert\right)\\
&\leq M(\|x_{0}\|+\beta r+\|g(0)\|)\\
&\quad + Mt_{1}(\|\mu_{1}\|_{L^{\infty}(I, \mathbb{R}^{+})} 
+ r\|\mu_{2}\|_{L^{\infty}(I, \mathbb{R}^{+})} 
+ K_{1}^{*}r\|\mu_{3}\|_{L^{\infty}(I, \mathbb{R}^{+})}
+\Vert B\Vert \Vert u\Vert)
\end{split}
\end{equation*}
for $t\in [0, t_1]$, and
\begin{equation*}
\begin{split}
\|(Q_{1}&x)(t) + (Q_{2}x)(t)\|\\
&\leq M\left(\|x(t^-_{k-1})\|+e_{k-1}\Vert x\Vert+\Vert D\Vert\|v(t_{k-1}^{-})\|\right)
+ M(t_{k}-t_{k-1})\\
&\quad \times \left(\|\mu_{1}\|_{L^{\infty}(I, \mathbb{R}^{+})} 
+ r\|\mu_{2}\|_{L^{\infty}(I, \mathbb{R}^{+})} + K_{k}^{*}r\|\mu_{3}\|_{L^{\infty}(I, 
\mathbb{R}^{+})}+\Vert B\Vert \Vert u\Vert\right)\\
&\leq M(\|x(t^-_{k-1})\|+e_{k-1}r+\Vert D\Vert\|v(t_{k-1}^{-})\|
+ M(t_{k}-t_{k-1}) \\
&\quad \times \left(\|\mu_{1}\|_{L^{\infty}(I, \mathbb{R}^{+})} 
+ r\|\mu_{2}\|_{L^{\infty}(I, \mathbb{R}^{+})} 
+ K_{k}^{*}r\|\mu_{3}\|_{L^{\infty}(I, \mathbb{R}^{+})}
+\Vert B\Vert \Vert u\Vert\right)
\end{split}
\end{equation*}
for $t\in (t_{k-1}, t_k]$. By the inequalities \eqref{e7}, 
we can find $\xi_1, \xi_2>0$ such that
\begin{equation*}
\|(Q_{1}x)(t) + (Q_{2}x)(t)\|
\leq
\begin{cases}
\xi_{1}, t\in [0, t_{1}],\\
\xi_{2}, t\in (t_{k-1},t_{k}], k=2,\ldots,m+1.
\end{cases}
\end{equation*}
Hence, $Q_{1}x+Q_{2}x$ is bounded. Now, let $x, y \in \Omega_{r}$. We have
\begin{equation*}
\|(Q_{1}x)(t) - (Q_{1}y)(t)\| 
\leq \|S_{q}(t)\| \|g(x)-g(y)\|
\leq M\beta \|x-y\|
\end{equation*}
for $t\in [0, t_{1}]$ and
\begin{equation*}
\begin{split}
\|&(Q_{1}x)(t) - (Q_{1}y)(t)\| \\
&\leq \|S_{q}(t-t_{k-1})\|[\Vert x(t^-_{k-1})-y(t^-_{k-1})\Vert
+\|I_{k-1}(x(t^-_{k-1}))-I_{k-1}(y(t^-_{k-1}))\|]\\
&\leq M\left[\Vert x(t^-_{k-1})-y(t^-_{k-1})\Vert+d_{k-1} \|x-y\|\right]
\end{split}
\end{equation*}
for $t\in (t_{k-1}, t_{k}]$, $k=2,\ldots,m+1$.
Since $M\beta<1$ and $M(1+d_{k-1})<1$, $k=2,\ldots,m+1$, 
it follows that $Q_{1}$ is a contraction mapping.
Let $\lbrace x_{n}\rbrace$ be a sequence in $\Omega_{r}$ 
such that $x_{n}\to x \in \Omega_{r}$. Since $f$ and $g$ are continuous, i.e.,
for all $\epsilon>0$, there exists a positive integer $n_0$, such that for $n>n_0$
$\|f (s, x_{n}(s), (Hx_{n})(s)) - f (s, x(s), (Hx)(s)) \|\leq\epsilon$ 
and $\|g(x_{n}) - g(x)\|\leq\epsilon$, the continuity of $I_i(x)$ 
on $(t_i, t_{i+1}]$ gives 
$\Vert I_i(x_n(t_i^-))-I_i(x(t_i^-))\Vert\leq \epsilon$,
$i=1,\ldots,m$. Now, for all $t \in[0, t_{1}]$,
\begin{align*}
&\|(Q_{2}x_{n})(t) - (Q_{2}x)(t)\|\\
&\leq\int_{0}^{t_{1}}\|T_{q}(t-\tau)\|\|BB^{*}T_{q}^{*}(t_{1}-\tau)
\mathcal{R}(\lambda, \Psi^{t_1}_{t_{0}, 1})\|
\Bigl[\|S_{q}(t_{1})(g(x_{n})-g(x))\|\\
&~~+\int_{0}^{t_{1}}\|
T_{q}(t_{1}-s)\|f (s, x_{n}(s), (Hx_{n})(s)) - f (s, x(s), (Hx)(s))\|ds\Bigr]d\tau\\
&~~+\int_{0}^{t_{1}}\|T_{q}(t-s)\|\|f (s, x_{n}(s), (Hx_{n})(s)) 
- f (s, x(s), (Hx)(s))\| ds\\
&\leq \dfrac{\epsilon}{\lambda}M^{3}\Vert B\Vert^2 t_{1}(2t_1+1).
\end{align*}
Moreover, for all $t\in (t_{i}, t_{i+1}]$, $i=1,\ldots,m$, one has
\begin{align*}
\|&(Q_{2}x_{n})(t) - (Q_{2}x)(t)\|\\
&\leq\int_{t_{i}}^{t}\|T_{q}(t-\tau)\|\|BB^{*}T_{q}^{*}(t_{i+1}
-\tau)\mathcal{R}(\lambda, \Psi^{t_{i+1}}_{t_{i}, 1})\|\\
&~~\times\biggl[\|S_{q}(t_{i+1}-t_{i})[x_n(t_i^-)-x(t_i^-)
+I_i(x_n(t_i^-))-I_i(x(t_i^-))]\|\\
&~~~+\int_{t_{i}}^{t_{i+1}}\|T_{q}(t_{i+1}-s)\|f (s, x_{n}(s), 
(Hx_{n})(s)) - f (s, x(s), (Hx)(s))\|ds)\biggr]d\tau\\
&~~+\int_{t_{i}}^{t}\|T_{q}(t-s)\|\|(f (s, x_{n}(s), 
(Hx_{n})(s)) - f (s, x(s), (Hx)(s))\| ds\\  
&\leq
\dfrac{2\epsilon}{\lambda}M^{3}\Vert B\Vert^2 
(t_{i+1}-t_i)(t_{i+1}-t_i+1).
\end{align*}
Therefore, $Q_{2}$ is continuous. Next, we prove the compactness 
of $Q_{2}$. For that, we first show that the set
$\left\lbrace (Q_{2}x)(t): x\in\Omega_{r}\right\rbrace$ 
is relatively compact in $PC(I, X)$.
By the assumptions of our theorem, we have
\begin{equation*}
\|(Q_{2}x)(t)\|
\leq Mt_{1}(\|\mu_{1}\|_{L^{\infty}(I, \mathbb{R}^{+})} 
+ r\|\mu_{2}\|_{L^{\infty}(I, \mathbb{R}^{+})} 
+ K_{1}^{*}r\|\mu_{3}\|_{L^{\infty}(I, \mathbb{R}^{+})}
+\Vert B\Vert \Vert u\Vert),
\end{equation*}
for $t\in [0, t_1]$, and
\begin{multline*}
\|(Q_{2}x)(t)\|
\leq 
M(t_{k}-t_{k-1})\\
\times \left(\|\mu_{1}\|_{L^{\infty}(I, \mathbb{R}^{+})} 
+ r\|\mu_{2}\|_{L^{\infty}(I, \mathbb{R}^{+})} + K_{k}^{*}r\|
\mu_{3}\|_{L^{\infty}(I, \mathbb{R}^{+})}
+\Vert B\Vert \Vert u\Vert\right), 
\end{multline*}
for $t\in (t_{k-1}, t_k]$, which gives the uniformly boundedness of 
$\left\lbrace (Q_{2}x)(t): x\in \Omega_{r}\right\rbrace$. 
We now show that $Q_{2}(\Omega_{r})$ is equicontinuous. 
Functions $\left\lbrace (Q_{2}x)(t): x\in \Omega_{r}\right\rbrace$ 
are equicontinuous at $t=0$. For any $x \in\Omega_{r}$, 
if $0 < r_{1} < r_{2} \leq t_{1}$, then
\begin{align*}
\|(Q_{2}&x)(r_{2}) - (Q_{2}x)(r_{1})\|\\
&\leq\int_{0}^{r_{1}}\Vert T_{q}(r_{2}-s)-T_{q}(r_{1}-s)
\Vert [\Vert Bu(s)\Vert + \Vert f(s, x(s), (Hx)(s))\Vert]ds\\
&\quad + \int_{r_{1}}^{r_{2}}\Vert T_{q}(r_{2}-s)\Vert [\Vert Bu(s)\Vert 
+ \Vert f(s, x(s), (Hx)(s))\Vert]ds \\
&\leq [r_1\Vert T_{q}(r_{2}-s)-T_{q}(r_{1}-s)\Vert +M(r_2-r_1)]\\
&\quad \times\left(\Vert B\Vert \Vert u\Vert+\|\mu_{1}\|_{L^{\infty}(I, 
\mathbb{R}^{+})} + r\|\mu_{2}\|_{L^{\infty}(I, \mathbb{R}^{+})} 
+ K_{1}^{*}r\|\mu_{3}\|_{L^{\infty}(I, \mathbb{R}^{+})}\right).
\end{align*}
Similarly, if $t_{i}< r_{1} < r_{2} \leq t_{i+1}$, then
\begin{align*}
\|(Q_{2}&x)(r_{2}) - (Q_{2}x)(r_{1})\|\\
&\leq\int_{t_i}^{r_{1}}\Vert T_{q}(r_{2}-s)-T_{q}(r_{1}-s)
\Vert [\Vert Bu(s)\Vert + \Vert f(s, x(s), (Hx)(s))\Vert]ds\\
&\quad + \int_{r_{1}}^{r_{2}}\Vert T_{q}(r_{2}-s)\Vert [\Vert Bu(s)\Vert 
+ \Vert f(s, x(s), (Hx)(s))\Vert]ds \\
&\leq [(r_1-t_i)\Vert T_{q}(r_{2}-s)-T_{q}(r_{1}-s)\Vert +M(r_2-r_1)]\\
&\quad \times\left(\Vert B\Vert \Vert u\Vert+\|\mu_{1}\|_{L^{\infty}(I, 
\mathbb{R}^{+})} + r\|\mu_{2}\|_{L^{\infty}(I, \mathbb{R}^{+})} 
+ K_{i+1}^{*}r\|\mu_{3}\|_{L^{\infty}(I, \mathbb{R}^{+})}\right).
\end{align*}
From $(H_1)$, it follows the continuity of operator 
$T_q(\cdot)$ in the uniform operator topology. 
Thus, the right hand side of the above inequality tends to zero as $r_{2}\to r_{1}$. 
Therefore, $\left\lbrace (Q_{2}x)(t): x\in\Omega_{r}\right\rbrace $ 
is a family of equicontinuous functions. According to the infinite dimensional 
version of the Ascoli--Arzela theorem, it remains to prove that, for any 
$t\in [0, b]\setminus\lbrace t_1,\ldots,t_m\rbrace$, the set 
$V(t):=\left\lbrace (Q_{2}x)(t): x \in \Omega_{r}\right\rbrace $ is relatively compact in
$PC(I, X)$. The case $t=0$ is trivial: $V(0)=\left\lbrace (Q_{2}x)(0)
: x(\cdot)\in \Omega_{r}\right\rbrace $ is compact in $PC(I, X)$. Let
$t\in (0, t_1]$ be a fixed real number and $h$ be a given real number satisfying 
$0< h <t_{1}$. Define $V_{h}(t) = \left\lbrace (Q_{2}^{h}x)(t)
: x\in \Omega_{r}\right\rbrace$,
\begin{align*}
(Q_{2}^{h}x)(t)&=\int_{0}^{t-h}T_{q}(t-s)Bu(s) ds +\int_{0}^{t-h}T_{q}(t-s)f(s, x(s),(Hx)(s))ds\\
&=T_{q}(h)\int_{0}^{t-h}T_{q}(t-s-h)Bu(s) ds\\
&\quad +T_{q}(h)\int_{0}^{t-h}T_{q}(t-s-h)f(s, x(s),(Hx)(s))ds\\
&=T_{q}(h)y_{1}(t,h).
\end{align*}
We use same arguments, we fix $t\in (t_i, t_{i+1}]$, 
and let $h$ be a given real number satisfying  $t_{i}< h <t_{i+1}$, 
we define $V_{h}(t) = \left\lbrace (Q_{2}^{h}x)(t): x\in \Omega_{r}\right\rbrace$,
\begin{align*}
(Q_{2}^{h}x)(t)&=\int_{t_{i}}^{t-h}T_{q}(t-s)Bu(s) ds 
+\int_{t_{i}}^{t-h}T_{q}(t-s)f(s, x(s),(Hx)(s))ds\\
&=T_{q}(h)\int_{t_{i}}^{t-h}T_{q}(t-s-h)Bu(s) ds\\
&+T_{q}(h)\int_{t_{i}}^{t-h}T_{q}(t-s-h)f(s, x(s),(Hx)(s))ds\\
&=T_{q}(h)y_{2}(t,h).
\end{align*}
The compactness of $T_{q}(h)$ in $PC(I, X)$, together with the boundedness 
of both $y_{1}(t, h)$ and $y_{2}(t, h)$ on $\Omega_{r}$, give the relativity 
compactness of the set $V_{h}(t)$ in $PC(I, X)$. Moreover, for all $t\in [0, t_{1}]$,
\begin{align*}
\|(Q_{2}&x)(t) - (Q_{2}^{h}x)(t)\|\\
& \leq\int_{t-h}^{t}T_{q}(t-s)Bu(s) ds
+\int_{t-h}^{t}T_{q}(t-s)f(s, x(s),(Hx)(s))ds\\
&\leq hM\left(\Vert B\Vert \Vert u\Vert+\|\mu_{1}\|_{L^{\infty}(I, \mathbb{R}^{+})} 
+ r\|\mu_{2}\|_{L^{\infty}(I, \mathbb{R}^{+})} 
+ K_{1}^{*}r\|\mu_{3}\|_{L^{\infty}(I, \mathbb{R}^{+})}\right).
\end{align*}
Also, for all $t\in (t_{i}, t_{i+1}]$,
\begin{align*}
\|(Q_{2}&x)(t) - (Q_{2}^{h}x)(t)\|\\
&\leq\int_{t-h}^{t}T_{q}(t-s)Bu(s) ds 
+\int_{t-h}^{t}T_{q}(t-s)f(s, x(s),(Hx)(s))ds\\
&\leq hM\left(\Vert B\Vert \Vert u\Vert+\|\mu_{1}\|_{L^{\infty}(I, \mathbb{R}^{+})} 
+ r\|\mu_{2}\|_{L^{\infty}(I, \mathbb{R}^{+})} 
+ K_{i+1}^{*}r\|\mu_{3}\|_{L^{\infty}(I, \mathbb{R}^{+})}\right).
\end{align*}
Choose $h$ small enough. It implies that there are relatively 
compact sets arbitrarily close to the set $V(t)$ 
for each $t\in [0, b]\setminus\lbrace t_1,\ldots,t_m\rbrace$. 
Then, $V(t)$, $t\in [0, b]\setminus\lbrace t_1,\ldots,t_m\rbrace$, 
is relatively compact in $PC(I, X)$. Since it is compact at $t=0$, we have 
the relatively compactness of $V(t)$ in $PC(I, X)$ for all 
$t\in [0, b]\setminus\lbrace t_1,\ldots,t_m\rbrace$.
Hence, by the Arzela--Ascoli theorem, we conclude that $Q_{2}$ is compact. 
From Lemma~\ref{Lemma:2.9}, we ensure that the control system \eqref{e1} 
has at least one mild solution on $t\in [0, b]\setminus\lbrace t_1,\ldots,t_m\rbrace$.
\end{proof}


\section{Approximate controllability}
\label{sec:4}

In this section, with help of the obtained existence theorem of mild solutions,
we show an approximate controllability result for system \eqref{e1}.

\begin{theorem}
\label{Theorem:2.11}
If (H$_1$)--(H$_5$) are satisfied and
$\lambda\mathcal{R}(\lambda, \Psi^{t_k}_{t_{k-1}, i})\rightarrow 0$ 
in the strong operator topology as $\lambda\rightarrow 0^{+}$, $i=1,2$, 
then the impulsive non-local fractional control system \eqref{e1}
is approximately controllable on $t\in [0, b]\setminus\lbrace t_1,\ldots,t_m\rbrace$.
\end{theorem}

\begin{proof}
According to Theorem~\ref{Theorem:2.10}, $Q_1^\lambda+Q_2^\lambda$ 
has a fixed point in $\Omega_{r}$ for any $\lambda>0$. This implies that there exists
$\overline{x}^{\lambda}\in (Q_1^{\lambda}
+Q_2^{\lambda})(\overline{x}^{\lambda})$ such that
\begin{equation*}
\overline{x}^{\lambda}(t)=
\begin{cases}
S_q(t)(x_0-\overline{g}^{\lambda}(\overline{x}^{\lambda}))\\
~~~~~+\int_{0}^{t}T_{q}(t-s)[\overline{f}^{\lambda}(s, \overline{x}^{\lambda}(s), 
(H\overline{x}^{\lambda})(s))+ B\overline{u}^{\lambda}(s)]ds,~ t\in [0, t_1],\\[0.3cm]
S_q(t-t_{k-1})[\overline{x}^{\lambda}(t_{k-1}^-)+\overline{I_{k-1}}^{\lambda}(
\overline{x}^{\lambda}(t_{k-1}^-))+D\overline{v}^{\lambda}(t_{k-1}^-)]\\
~~~~~+\int_{t_{k-1}}^{t}T_{q}(t-s)[\overline{f}^{\lambda}(s, 
\overline{x}^{\lambda}(s), (H\overline{x}^{\lambda})(s))
+ B\overline{u}^{\lambda}(s)]ds,~  t\in (t_{k-1}, t_{k}],
\end{cases}
\end{equation*}
where for $t\in [0, t_1]$ we have
\begin{multline*}
\overline{u}^{\lambda}
= B^{\ast}T_q^{\ast}(t_1-t)
\mathcal{R}(\lambda, \Psi^{t_1}_{0,1})\biggl[
x_1-S_q(t_1)(x_0-\overline{g}^{\lambda}(\overline{x}^{\lambda}))\\
-\int_{0}^{t_1}T_q(t_1-s)\overline{f}^{\lambda}(s, \overline{x}^{\lambda}(s), 
(H\overline{x}^{\lambda})(s))ds\biggr]
\end{multline*}
while for $k=2,\ldots,m+1$ 
\begin{multline*}
\overline{u}^{\lambda} =
B^{\ast}T_q^{\ast}(t_k-t)
\mathcal{R}(\lambda, \Psi^{t_k}_{t_{k-1},1})\biggl[x_k-S_q(t_k-t_{k-1})[
\overline{x}^{\lambda}(t_{k-1}^-)+\overline{I_{k-1}}^{\lambda}(
\overline{x}^{\lambda}(t_{k-1}^-))]\\
-\int_{t_{k-1}}^{t_k}T_q(t_k-s)
\overline{f}^{\lambda}(s, \overline{x}^{\lambda}(s), 
(H\overline{x}^{\lambda})(s))ds\biggr]
\end{multline*}
and
\begin{multline*}
\overline{v}^{\lambda}=
D^{\ast}S_q^{\ast}(t_k-t_{k-1})
\mathcal{R}(\lambda, \Psi^{t_k}_{t_{k-1},2})\\
\times \biggl[
x_k-S_q(t_k-t_{k-1})[\overline{x}^{\lambda}(t_{k-1}^-)
+\overline{I_{k-1}}^{\lambda}(\overline{x}^{\lambda}(t_{k-1}^-))]\\
-\int_{t_{k-1}}^{t_k}T_q(t_k-s)\overline{f}^{\lambda}(s, 
\overline{x}^{\lambda}(s), (H\overline{x}^{\lambda})(s))ds\biggr].
\end{multline*}
Furthermore, 
\begin{multline*}
\overline{x}^{\lambda}(t_1)=
S_q(t_1)(x_0-\overline{g}^{\lambda}(\overline{x}^{\lambda}))\\
+\int_{0}^{t_1}T_q(t_1-s)[\overline{f}^{\lambda}(s, 
\overline{x}^{\lambda}(s), (H\overline{x}^{\lambda})(s))
+ B\overline{u}^{\lambda}(s)]ds,
\end{multline*}
\begin{multline*}
\overline{x}^{\lambda}(t_k)=
S_q(t_k-t_{k-1})[\overline{x}^{\lambda}(t_{k-1}^-)
+\overline{I_{k-1}}^{\lambda}(\overline{x}^{\lambda}(t_{k-1}^-))
+D\overline{v}^{\lambda}(t_{k-1}^-)]\\
+\int_{t_{k-1}}^{t_k}T_q(t_k-s)[\overline{f}^{\lambda}(s, 
\overline{x}^{\lambda}(s), (H\overline{x}^{\lambda})(s))
+ B\overline{u}^{\lambda}(s)]ds,
\end{multline*}
$k=2,\dots,m+1$, with
\begin{multline*}
x_{t_1}-\overline{x}^{\lambda}(t_1)=
x_1-\Psi^{t_{1}}_{0, 1}\mathcal{R}(\lambda, 
\Psi^{t_{1}}_{0, 1})\biggl\lbrace x_1-S_q(t_1)(x_0
-\overline{g}^{\lambda}(\overline{x}^{\lambda}))\\
-\int_{0}^{t_1}T_q(t_1-s)\overline{f}^{\lambda}(s, 
\overline{x}^{\lambda}(s), (H\overline{x}^{\lambda})(s))ds\biggr\rbrace\\
-S_q(t_1)(x_0-\overline{g}^{\lambda}(\overline{x}^{\lambda}))
-\int_{0}^{t_1}T_q(t_1-s)\overline{f}^{\lambda}(s, \overline{x}^{\lambda}(s), 
(H\overline{x}^{\lambda})(s))ds,
\end{multline*}
\begin{equation*}
\begin{split}
x_{t_k}&-\overline{x}^{\lambda}(t_k)=
x_k-\Psi^{t_{k}}_{k-1, 2}\mathcal{R}(\lambda, \Psi^{t_{k}}_{k-1, 2})\\
&\times \biggl\lbrace
x_k-S_q(t_k-t_{k-1})\left[\overline{x}^{\lambda}(t_{k-1}^-)
+\overline{I_{k-1}}^{\lambda}(\overline{x}^{\lambda}(t_{k-1}^-))\right]\\
&-\int_{t_{k-1}}^{t_k}
T_q(t_k-s)\overline{f}^{\lambda}\left(s, \overline{x}^{\lambda}(s), 
(H\overline{x}^{\lambda})(s)\right)ds\biggr\rbrace\\
&-S_q(t_k-t_{k-1})\left[\overline{x}^{\lambda}(t_{k-1}^-)
+\overline{I_{k-1}}^{\lambda}(\overline{x}^{\lambda}(t_{k-1}^-))\right]\\
&-\int_{t_{k-1}}^{t_k}
T_q(t_k-s)\overline{f}^{\lambda}\left(s, \overline{x}^{\lambda}(s), 
(H\overline{x}^{\lambda})(s)\right)ds\\
&-\Psi^{t_{k}}_{k-1, 1}\mathcal{R}(\lambda, 
\Psi^{t_{k}}_{k-1, 1})\biggl\lbrace x_k
-S_q(t_k-t_{k-1})\left[\overline{x}^{\lambda}(t_{k-1}^-)
+\overline{I_{k-1}}^{\lambda}(\overline{x}^{\lambda}(t_{k-1}^-))\right]\\
&-\int_{t_{k-1}}^{t_k}T_q(t_k-s)
\overline{f}^{\lambda}\left(s, \overline{x}^{\lambda}(s), 
(H\overline{x}^{\lambda})(s)\right)ds\biggr\rbrace,
\quad  k=2,\dots,m+1.
\end{split}
\end{equation*}
From \eqref{e5} we have
$I-\Psi^{t_k}_{t_{k-1},i}\mathcal{R}\left(\lambda, \Psi^{t_k}_{t_{k-1},i}\right)
=\lambda\mathcal{R}\left(\lambda, \Psi^{t_k}_{t_{k-1},i}\right)$, $i=1,2$,
and
\begin{multline}
\label{e8a}
x_{t_1}-\overline{x}^{\lambda}(t_1)=
\lambda\mathcal{R}\left(\lambda, \Psi^{t_{1}}_{0, 1}\right)\biggl\lbrace
x_1-S_q(t_1)(x_0-\overline{g}^{\lambda}(\overline{x}^{\lambda}))\\
-\int_{0}^{t_1}T_q(t_1-s)\overline{f}^{\lambda}(s, \overline{x}^{\lambda}(s), 
(H\overline{x}^{\lambda})(s))ds\biggr\rbrace,
\end{multline}
\begin{multline}
\label{e8b}
x_{t_k}-\overline{x}^{\lambda}(t_k)=
\lambda\left[\mathcal{R}\left(\lambda, \Psi^{t_{k}}_{k-1, 1}\right)
+\mathcal{R}\left(\lambda, \Psi^{t_{k}}_{k-1, 2}\right)\right]\\
\times \biggl\lbrace
x_k-S_q(t_k-t_{k-1})\left[\overline{x}^{\lambda}(t_{k-1}^-)
+\overline{I_{k-1}}^{\lambda}(\overline{x}^{\lambda}(t_{k-1}^-))\right]\\
-\int_{t_{k-1}}^{t_k}T_q(t_k-s)\overline{f}^{\lambda}(s, 
\overline{x}^{\lambda}(s), (H\overline{x}^{\lambda})(s))]ds\biggr\rbrace,
~  k=2,\dots,m+1.
\end{multline}
Since compactness of both $S_q(t)_{t>0}$ and $T_q(t)_{t>0}$ hold, 
and also boundedness of $\overline{f}^{\lambda}$, $\overline{g}^{\lambda}$ 
and $\overline{I_{k-1}}^{\lambda}$, we can use on \eqref{e8a}--\eqref{e8b} 
the fact that $\lambda\mathcal{R}(\lambda, \Psi^{t_k}_{t_{k-1}, i})\rightarrow 0$ 
in the strong operator topology as $\lambda\rightarrow 0^{+}$, $i=1,2$. 
This gives $\Vert x_{t_k}-\overline{x}^{\lambda}(t_k)\Vert_\alpha\to 0$ 
as $\lambda\rightarrow 0^{+}$, $i=1,2$. Hence, 
the impulsive non-local fractional control system \eqref{e1}
is approximately controllable on $t\in [0, b]\setminus\lbrace t_1,\ldots,t_m\rbrace$.
\end{proof}


\section{Optimality}
\label{sec:5}

Let $Y$ be a separable reflexive Banach space and $w_{f}(Y)$ represent a class
of non-empty, closed and convex subsets of $Y$. The multifunction 
$w: I\longrightarrow w_{f}(Y)$ is measurable and $w(\cdot)\subset E$, 
where $E$ is a bounded set of $Y$. We give the admissible control set as follows:
$$
U_{ad}=\left\lbrace (u, v)\in L^{1}(E)\times L^{1}(E)| 
u(t), v(t)\in w(t) ~a. e.\right\rbrace \neq\emptyset.
$$
Consider the following impulsive nonlocal fractional control system:
\begin{equation} 
\label{e9}
\left\lbrace 
\begin{array}{lll} 
^{C}D_{t}^{q}x(t)=Ax(t)+f(t, x(t), (Hx)(t))+\mathcal{B}u(t), 
t\in (0, b]\setminus\lbrace t_1, t_2,\ldots, t_m\rbrace,\\
x(0)+g(x)=x_{0}\in X,\\
\triangle x(t_{i})=I_{i}(x(t_{i}^{-}))+\mathcal{D}v(t_{i}^{-}), i=1,2,\ldots,m,
\quad (u, v)\in U_{ad},
\end{array}
\right.
\end{equation}
where $\mathcal{B}, \mathcal{D}\in L^{\infty}(I, L(Y, X))$. 
It is clear that $\mathcal{B}u, \mathcal{D}v\in L^{1}(I, X)$ 
for all $(u, v)\in U_{ad}$.
Let $x^{u,v}$ be a mild solution of system \eqref{e9} 
corresponding to controls $(u, v)\in U_{ad}$. We consider
the Bolza problem $(BP)$: find an optimal triplet 
$(x^{0}, u^{0}, v^{0})\in PC(I, X)\times U_{ad}$ such that
$\mathcal{J}(x^{0}, u^{0}, v^{0})\leq \mathcal{J}(x^{u,v}, u, v)$,  
for all $(u, v)\in U_{ad}$, where
$$
\mathcal{J}(x^{u,v}, u, v)=\sum\limits_{i=1}^{m+1}\left[\Phi (x^{u,v}(t_i))
+\int_{t_{i-1}}^{t_i} \mathcal{L}(t, x^{u,v}(t), u(t), v(t))dt\right], 
$$
$i=1,\ldots,m+1$. The following extra assumptions are needed:
\begin{itemize}
\item[(H$_6$)] The functional $\mathcal{L}: I\times X\times Y^2
\to \mathbb{R}\cup\left\lbrace\infty\right\rbrace$ is Borel measurable.

\item[(H$_7$)] $\mathcal{L}(t,\cdot,\cdot,\cdot)$ is sequentially 
lower semi-continuous on $X\times Y^2$, a.e. on $I$.

\item[(H$_8$)] $\mathcal{L}(t,\cdot,\cdot,\cdot)$ is convex 
on $Y^2$ for each $x\in X$ and almost all $t\in I$.

\item[(H$_9$)] There is a non-negative function $\varphi
\in L^{\infty}(I, \mathbb{R})$ and $c_1, c_2, c_3\geq0$ such that
$\mathcal{L}(t, x, u, v)\geq\varphi(t) + c_1\|x\| + c_2\|u\|_{Y}^{p} + c_3\|v\|_{Y}^{p}$.

\item[(H$_{10}$)] The functional $\Phi: X\to \mathbb{R}$ 
is continuous and non-negative.
\end{itemize}

\begin{theorem}
\label{Theorem:2.12}
If (H$_6$)--(H$_{10}$) hold together with the assumptions of 
Theorem~\ref{Theorem:2.10}, then the Bolza problem $(BP)$ 
admits at least one optimal triplet on $PC\times U_{ad}$.
\end{theorem}

\begin{proof}
Assume $\inf \left\lbrace \mathcal{J}(x^{u,v}, u, v)\vert (u, v)
\in U_{ad}\right\rbrace = \delta<+\infty $. From (H$_6$)--(H$_{10}$), 
\begin{align*}
\mathcal{J}&(x^{u,v}, u, v)\\
&\geq\sum\limits_{i=1}^{m+1}\left[\Phi (x^{u,v}(t_i))
+ \int_{t_{i-1}}^{t_i}\left\lbrace\varphi(t) + c_1\|x(t)\| 
+ c_2\|u(t)\|_{Y}^{p} + c_3\|v(t)\|_{Y}^{p}\right\rbrace dt\right]\\
&\geq -\eta>-\infty,~ i=1,\ldots,m+1.
\end{align*}
Here, $\eta$ is a positive constant, i.e.,
$\delta\geq -\eta> -\infty$. 
By the definition of infimum, there exists 
a minimizing sequence of feasible triplets
$\left\lbrace(x^{n}, u^{n}, v^{n})\right\rbrace\subset 
\mathcal{A}_{ad}$, where $\mathcal{A}_{ad} \equiv\lbrace (x, u, v)\vert$ $x$ 
is a mild solution of system \eqref{e9} corresponding to $(u, v)\in U_{ad}\rbrace$, 
such that $\mathcal{J}(x^{n}, u^{n}, v^{n}) \to \delta$ as $m \to +\infty$. 
As $\left\lbrace (u^{n}, v^{n})\right\rbrace \subseteq U_{ad}$ and  
$\left\lbrace u^{n}, v^{n}\right\rbrace $ is bounded in $L^{1}(I, Y)$, 
then there exists a subsequence, still denoted by $\left\lbrace (u^{n}, 
v^{n})\right\rbrace$, and $u^{0}, v^{0} \in L^{1}(I, Y)$, such that
$(u^n, v^n)\stackrel{\text{weakly}}{\longrightarrow}(u^0, v^0)$ 
in $L^{1}(I, Y)\times L^{1}(I, Y)$.
Since the admissible control set $U_{ad}$ is convex and closed, 
by Marzur lemma, we have $(u^{0}, v^{0})\in U_{ad}$.
Suppose that $x^{n}$ is a mild solution of system \eqref{e9}, 
corresponding to $u^{n}$ and $v^{n}$, that satisfies
\begin{equation*}
x^n(t)= 
S_{q}(t)(x_{0}-g(x^n))+\int_{0}^{t}T_{q}(t-s)(f(s,x^n(s),(Hx^n)(s))
+\mathcal{B}u^n(s))ds,
\end{equation*}
for $t\in[0, t_{1}]$, and
\begin{multline*}
x^n(t)= 
S_{q}(t-t_{i})[x^n(t_i^-)+I_{i}(x^n(t_{i}^{-}))+\mathcal{D}v^n(t_{i}^{-})]\\
+\int_{t_{i}}^{t}T_{q}(t-s)[f(s,x^n(s),(Hx^n)(s))+\mathcal{B}u^n(s)]ds
\end{multline*}
for $t\in(t_{i}, t_{i+1}]$, $i=1,\ldots,m$.
From (H$_2$), the non-linear function $f$ is bounded and continuous. 
Then, there exists a subsequence (with the same notation)
$\left\lbrace f(s, x^{n}, (Hx^{n})(s))\right\rbrace $ and 
$f(s, x^0, (Hx^0)(s)) \in L^{1}(I, X)$ such that $f(s, x^{n}, (Hx^{n})(s))$  
converges weakly to $f(s, x^0, (Hx^0)(s))$. Also, the same arguments on 
(H$_3$) and (H$_5$) yield other weak convergences of $g(x^n)$ 
and $I_i(x^n)$ to $g(x^0)$ and $I_i(x^0)$, respectively. Let us denote
\begin{equation*}
(P_1x)(t)=S_{q}(t)g(x)+\int_{0}^{t}T_{q}(t-s)(f(s,x(s),(Hx)(s))
+\mathcal{B}u(s))ds, \quad t\in[0, t_{1}],
\end{equation*}
\begin{multline*}
(P_2x)(t)=S_{q}(t-t_{i})[x(t_i^-)+I_{i}(x(t_{i}^{-}))+\mathcal{D}v(t_{i}^{-})]\\
+\int_{t_{i}}^{t}T_{q}(t-s)[f(s,x(s),(Hx)(s))+\mathcal{B}u(s)]ds,
\ t\in(t_{i}$, $t_{i+1}], \ i=1,\ldots,m.
\end{multline*}
Obviously, $(P_1x)(t)$ and $(P_2x)(t)$ are strongly continuous operators. 
Thus, $(P_1x^n)(t)$ and $(P_2x^n)(t)$ strongly converge to $(P_1x)(t)$ 
and $(P_2x)(t)$, respectively. Next, we consider the system
\begin{equation*}
x^0(t)=
S_{q}(t)(x_{0}-g(x^0))+\int_{0}^{t}T_{q}(t-s)(f(s,x^0(s),(Hx^0)(s))
+\mathcal{B}u^0(s))ds,
\end{equation*}
$t\in[0, t_{1}]$, and
\begin{multline*}
x^0(t)=
S_{q}(t-t_{i})[x^0(t_i^-)+I_{i}(x^0(t_{i}^{-}))+\mathcal{D}v^0(t_{i}^{-})]\\
+\int_{t_{i}}^{t}T_{q}(t-s)[f(s,x^0(s),(Hx^0)(s))+\mathcal{B}u^0(s)]ds, 
\end{multline*}
$t\in(t_{i}, t_{i+1}]$, $i=1,\ldots,m$.
It is not difficult to check that $\Vert x^n(t)-x^0(t)\Vert\to 0$ as $n \to\infty$.
Therefore, we can infer that $x^{n}$ strongly converges 
to $x^0$ in $PC(I, X)$ as $n\to\infty$. From assumptions (H$_6$)--(H$_{10}$) 
and Balder's theorem, we get
\begin{align*}
\eta &=\lim\limits_{n\to\infty}\sum\limits_{i=1}^{m+1}\left[\Phi (x^{n}(t_i))
+\int_{t_{i-1}}^{t_i} \mathcal{L}(t, x^{n}(t), u^n(t), v^n(t))dt\right]\\
&\geq\sum\limits_{i=1}^{m+1}\left[\Phi (x^0(t_i))+\int_{t_{i-1}}^{t_i} 
\mathcal{L}(t, x^0(t), u^0(t), v^0(t))dt\right]\\
&=\mathcal{J}(x^{0}, 
u^{0}, v^{0})\geq \eta,~ i=1,\ldots,m+1,
\end{align*}
which implies that $\mathcal{J}$ attains its minimum 
at $(x^0, u^{0}, v^{0})\in PC(I, X)\times U_{ad}$. 
\end{proof}


\section*{Acknowledgement} 

This research is part of first author's PhD project.
Debbouche and Torres are supported by FCT and CIDMA 
through project UID/MAT/04106/2013.



\end{document}